\documentclass[oneside,english]{amsart}
\usepackage[T1]{fontenc}
\usepackage[latin9]{inputenc}
\usepackage{color}
\usepackage{float}
\usepackage{mathtools}
\usepackage{amstext}
\usepackage{amsthm}
\usepackage{amssymb}
\usepackage{stmaryrd}
\usepackage{graphicx}
\PassOptionsToPackage{normalem}{ulem}
\usepackage{ulem}

\makeatletter

\newcommand*\LyXZeroWidthSpace{\hspace{0pt}}
\newcommand{\lyxmathsym}[1]{\ifmmode\begingroup\def\b@ld{bold}
  \text{\ifx\math@version\b@ld\bfseries\fi#1}\endgroup\else#1\fi}

\numberwithin{equation}{section}
\theoremstyle{plain}
\newtheorem{thm}{\protect\theoremname}[section]
\theoremstyle{definition}
\newtheorem{defn}[thm]{\protect\definitionname}
\theoremstyle{remark}
\newtheorem{rem}[thm]{\protect\remarkname}
\theoremstyle{plain}
\newtheorem{prop}[thm]{\protect\propositionname}
\theoremstyle{plain}
\newtheorem*{prop*}{\protect\propositionname}
\theoremstyle{plain}
\newtheorem{cor}[thm]{\protect\corollaryname}
\theoremstyle{plain}
\newtheorem*{thm*}{\protect\theoremname}
\theoremstyle{definition}
\newtheorem{example}[thm]{\protect\examplename}
\theoremstyle{remark}
\newtheorem*{acknowledgement*}{\protect\acknowledgementname}

\makeatother

\usepackage{babel}
\providecommand{\acknowledgementname}{Acknowledgement}
\providecommand{\corollaryname}{Corollary}
\providecommand{\definitionname}{Definition}
\providecommand{\examplename}{Example}
\providecommand{\propositionname}{Proposition}
\providecommand{\remarkname}{Remark}
\providecommand{\theoremname}{Theorem}

\begin{document}
\title[Transformation Theorem for Signature-Type Changing Manifolds ]{A Transformation Theorem for Transverse Signature-Type Changing Semi-Riemannian Manifolds}
\author{W. Hasse \& N. E. Rieger}
\address{W. Hasse, Institute for Physics and Astronomy, Technical University
Berlin, Hardenbergstr. 36, 10623 Berlin, Germany, and Wilhelm Foerster
Observatory Berlin, Munsterdamm 90, 12169 Berlin, Germany }
\address{N. E. Rieger, Department of Mathematics, University of Zurich, Winterthurer\-strasse
190, 8057 Zurich, Switzerland}
\curraddr{Mathematics Department, Yale University, 219 Prospect Street, New
Haven, CT 06520, USA}
\email{n.rieger@math.uzh.ch}
\begin{abstract}
In the early eighties Hartle and Hawking put forth that signature-type
change may be conceptually interesting, paving the way to the so-called
``no boundary'' proposal for the initial
conditions for the universe. Such singu\-larity-free universes have
no beginning, but they do have an origin of time. In mathematical
terms, we are dealing with signature-type changing manifolds where
a Riemannian region is smoothly joined to a Lorentzian region at the surface of transition
where time begins. We present a transformation prescription to transform an arbitrary
 Lorentzian manifold into a singular signature-type changing manifold.
Then we establish the Transformation Theorem, asserting that, conversely,
under certain conditions, such a metric $(M,\tilde{g})$ can be obtained
from some Lorentz metric $g$ through the aforementioned transformation
procedure. By augmenting the assumption by certain constraints, mutatis
mutandis, the global version of the Transformation Theorem can be
proven as well. 
In conclusion, we make use of the Transformation Prescription to demonstrate
that the induced metric on the hypersurface of signature change is
either Riemannian or a positive semi-definite pseudo metric.
\end{abstract}

\maketitle
\address{Corresponding author:}

\address{N. E. Rieger, Department of Mathematics, Yale University, USA}

\email{n.rieger@yale.edu}
~\\ 

\section{Introduction }

In 1983 Hartle and Hawking~\cite{Hartle Hawking - Wave function of the Universe}
proposed that signature-type change may be conceptually intriguing,
leading to the so-called ``no boundary''
proposal for the initial conditions for the universe. According to
the Hartle-Hawking proposal the universe has no beginning because
there is neither a singularity nor a boundary to the spacetime (although
singularities can be considered points where curves stop at finite
parameter value, a general definition can be deemed to  be difficult~\cite{Geroch - What is a singularity in general relativity}).
In such singularity-free universes, there is no distinct beginning,
but they do possess an origin of time. While from a physical viewpoint,
a signature-type change may be used to avoid the singularities of
general relativity (e.g. the big bang can be replaced by a Riemannian
domain prior to the emergence of time), from a mathematical point
of view, the ``no boundary'' proposal
can be taken as ``no boundary'' condition
for manifolds~\cite{Gibbons + Hartle,Halliwell + Hartle}. 

~

It is essential to note that a signature-changing metric is inherently
degenerate or discontinuous at the boundary of the hypersurface~\cite{Dray - Gravity and Signature Change}.
We allow the metric to become degenerate, thus defining a singular
semi-Riemannian manifold as a generalization of a semi-Riemannian
manifold where the metric can become degenerate (singular). Hence,
in the present article we will discuss singular semi-Riemannian manifolds
for which the metric constitutes a smooth $(0,2)$-tensor field that
is degenerate at a subset $\mathcal{H}\subset M$, where the bilinear
type of the metric changes upon crossing $\mathcal{H}$. 

~

In this article we present a procedure, called the Transformation
Prescription, to transform an arbitrary Lorentzian manifold into a
singular transverse signature-type changing manifold. Then we prove
the so-called Transformation Theorem saying that locally the metric
$\tilde{g}$ associated with a transverse signature-type changing
manifold $(M,\tilde{g})$ is equivalent to the metric obtained from
some Lorentzian metric $g$ via the aforementioned Transformation
Prescription. By augmenting the assumption by certain constraints,
mutatis mutandis, the global version of the Transformation Theorem
is proven as well.

~

In conclusion, we make use of the Transformation Prescription to demonstrate
that the induced metric on the hypersurface $\mathcal{H}$ is either
Riemannian or a positive semi-definite pseudo metric.

\subsection{Transverse type-changing singular semi-Riemannian manifold}

Unless otherwise specified, the considered manifolds, denoted as $M$
with dimension $\dim(M)=n$, are assumed to be locally homeomorphic
to $\mathbb{R}^{n}$. Furthermore, these manifolds are expected to
be connected, second countable, and Hausdorff. This definition also
implies that all manifolds have empty boundary. Furthermore, we will
usually assume that the manifolds we consider are smooth as well.
If not indicated otherwise, all associated structures and geometric
objects (curves, maps, fields, differential forms etc.) are assumed
to be smooth. 
\begin{defn}
A \textit{singular semi-Riemannian manifold} is a generalization of
a semi-Riemannian manifold. It is a differentiable manifold having
on its tangent bundle a symmetric bilinear form which is allowed to
become degenerate.
\end{defn}

\vspace{0.001\paperheight}

\begin{defn}
\label{Definition. signature change}Let $(M,g)$ be a singular semi-Riemannian
manifold and let be $p\in M$. We say that the metric changes its
signature at a point $p\in M$ if any neighborhood of $p$ contains
at least one point $q$ where the metric's signature differs from
that at $p$.
\end{defn}

We side with Kossowski and Kriele~\cite{Kossowksi + Kriele - Signature type change and absolute time in general relativity}
and require that $(M,g)$ should be a semi-Riemannian manifold of
$\dim M\geq2$ with $g$ being a smooth, symmetric, degenerate $(0,2)$-tensor
on $M$, and $\mathcal{H}:=\{q\in M\!\!:g\!\!\mid_{q} is\;degenerate\}$.{\small{}
}In addition, we assume that one connected component, $M_{R}$, of
$M\setminus\mathcal{H}$ is Riemannian, while all other connected
components $(M_{L_{\alpha}})_{\alpha\in I}\subseteq M_{L}\subset M$
are Lorentzian. What is more, we assume throughout that the point
set $\mathcal{H}$ where $g$ becomes degenerate, is not empty, indicating
that $\mathcal{H}$ is the locus where the rank of $g$ fails to be
maximal.

~\\ Moreover, we impose the following two conditions~\cite{Kossowksi + Kriele - Signature type change and absolute time in general relativity}:
\begin{enumerate}
\item We call the metric $g$ a codimension-$1$ \textit{transverse type-changing
metric} if $d(\det([g_{\mu\nu}]))_{q}\neq0$ for any $q\in\mathcal{H}$
and any local coordinate system $\xi=(x^{0},\ldots,x^{n-1})$ around
$q$. Then we call $(M,g)$ a \textit{transverse type-changing singular
semi-Riemannian manifold}~\cite{Aguirre - Transverse Riemann-Lorentz type-changing metrics with tangent radical,Kossowksi + Kriele - Signature type change and absolute time in general relativity}.~\\ This
implies that the subset $\mathcal{H}\subset M$ is a smoothly embedded
hypersurface in $M$, and the bilinear type of $g$ changes upon crossing
$\mathcal{H}$. Moreover, at every point $q\in\mathcal{H}$, there
exists a one-dimensional subspace $\textrm{Rad}_{q}\subset T_{q}M$ (referred
to as the radical at $q\in\mathcal{H}$) that is orthogonal to all
of $T_{q}M$.~\\{}
\item \textbf{The radical $\textrm{Rad}_{q}$ is transverse} to $\mathcal{H}$ for any $q\in\mathcal{H}$.
Henceforward, we assume throughout that $(M,g)$ is a singular transverse
type-changing semi-Riemannian manifold with a \textit{transverse radical},
unless explicitly stated otherwise.
\end{enumerate}
~
\begin{rem}
The radical at $q\in\mathcal{H}$ is defined as the subspace $\textrm{Rad}_{q}:=\{w\in T_{q}M:g(w,\centerdot)=0\}$.
This means $g(v_{q},\centerdot)=0$ for all $v_{q}\in \textrm{Rad}_{q}$. Note
that the radical can be either transverse or tangent to the hypersurface
$\mathcal{H}$. The radical $\textrm{Rad}_{q}$ is called \textit{transverse}~\cite{Kossowski - The Volume BlowUp and Characteristic Classes for Transverse}
if $\textrm{Rad}_{q}$ and $T_{q}\mathcal{H}$ span $T_{q}M$ for any $q\in\mathcal{H}$,
i.e. $\textrm{Rad}_{q}\oplus T_{q}\mathcal{H}=T_{q}M$. This means that $\textrm{Rad}_{q}$
is not a subset of $T_{q}\mathcal{H}$, and $\textrm{Rad}_{q}$ is obviously
not tangent to $\mathcal{H}$ for any $q$.
\end{rem}

\subsection{Statement of results}

We state the first result as a proposition, called the \textit{Transformation
Prescription}, where we present a procedure to transform an arbitrary
Lorentzian manifold into a signature-type changing manifold. 

\begin{prop}[Transformation Prescription]
\label{Proposition Transformation-Prescription-1}
Let $(M,g)$ be an (not necessarily time-orientable) Lorentzian manifold
of $\dim n\geq2$. Then we obtain a signature-type changing metric
$\tilde{g}$ via the Transformation Prescription $\tilde{g}=g+fV^{\flat}\otimes V^{\flat}$,
where $f\colon M\longrightarrow\mathbb{R}$ is a smooth transformation
function and $V$ is one of the unordered pair $\{V,\lyxmathsym{\textemdash}V\}$
of a global smooth non-vanishing line element field.
\end{prop}

Our main theorem, the so-called \textit{Transformation Theorem}, makes
the stronger assertion that, locally, the metric $\tilde{g}$ associated
with a signature-type changing manifold $(M,\tilde{g})$ is equivalent
to the metric obtained from a Lorentzian metric $g$ via the aforementioned
Transformation Prescription. 

\begin{thm}[Local Transformation Theorem, transverse radical]
\label{Transformation-Theorem-(local)-1}
For every $q\in M$ there exists a neighborhood $U(q)$, such that
the metric $\tilde{g}$ associated with a signature-type changing
manifold $(M,\tilde{g})$ is a transverse, type-changing metric with
a transverse radical if and only if $\tilde{g}$ is locally obtained
from a Lorentzian metric $g$ via the Transformation Prescription~\ref{Proposition Transformation-Prescription-1}
$\tilde{g}=g+fV^{\flat}\otimes V^{\flat}$, where, for all $q\in \mathcal H \coloneqq f^{-1}(1)=\{p\in M\colon f(p)=1\}$,
\[
df(q)\neq0 \quad \text{and} \quad (df(V))(q)\neq0.
\]
\end{thm}

By augmenting the assumptions by an additional constraint, mutatis
mutandis, the global version of the Transformation Theorem is proven
as well. Provided the Riemannian sector with boundary $M_{R}\cup\mathcal{H}$
possess a smoothly defined non-vanishing line element field that is
transverse to the boundary $\mathcal{H}$, then we are able to show 

\begin{thm}[Global Transformation Theorem, transverse radical]
\label{Transformation-Theorem-(global)-1}
Let $M$ be an transverse, signature-type changing manifold of $\dim(M)=n\geq2$,
which admits in $M_{R}\cup\mathcal{H}$ a smoothly defined non-vanishing
line element field that is transverse to the boundary $\mathcal{H}$.
Then the metric $\tilde{g}$ associated with a signature-type changing
manifold $(M,\tilde{g})$ is a transverse, type-changing metric with
a transverse radical if and only if $\tilde{g}$ is obtained from
a Lorentzian metric $g$ via the Transformation Prescription $\tilde{g}=g+fV^{\flat}\otimes V^{\flat}$,
where, for all $q\in \mathcal H \coloneqq f^{-1}(1)=\{p\in M\colon f(p)=1\}$,
\[
df(q)\neq0 \quad \text{and} \quad (df(V))(q)\neq0.
\]
\end{thm}

Ultimately, we make use of the Transformation Prescription to demonstrate
that the induced metric on the hypersurface of signature change is
either Riemannian or a positive semi-definite pseudo-metric. Specifically,
the induced metric is Riemannian if the radical is transverse to the
hypersurface, and the metric is a positive semi-definite pseudo-metric
if the radical is tangent to the hypersurface.

\begin{thm}
\label{par: Theorem Rad_q-1}If $q\in\mathcal{H}$ and $x\notin \textrm{Rad}_{q}$, 
then $\tilde{g}(x,x)>0$ holds for all $x\in T_{q}M$.
\end{thm}

We conclude from this result, that dependent on whether $\textrm{Rad}_{q}\subset T_{q}\mathcal{H}$
or $\textrm{Rad}_{q}\subset\ker(df_{q})=T_{q}\mathcal{H}$, the induced metric
on the hypersurface of signature change $\mathcal{H}$ can be either
Riemannian (and non-degenerate) or a positive semi-definite pseudo-metric
with signature $(0,\underset{(n-2)\,\text{times}}{\underbrace{+,\ldots,+}})$.
The latter one is degenerate if $\ker(df_{q})=T_{q}\mathcal{H}=\textrm{Rad}_{q}$.

\section{Transformation Prescription\label{sec:Transformation-Prescription}}

Let $M$ be an $n$-dimensional, not necessarily time-orientable Lorentzian
manifold $(M,g)$ with a smooth Lorentzian metric $g$. In general,
the existence of a global non-vanishing, timelike $C^{\infty}$ vector
field is not guaranteed, it does exist under ideal circumstances.
However, we can alternatively depend on the following fact~\cite{Markus - Line-Element Fields and Lorentz Structures on Differentiable Manifolds}:

~

A line element field $\{V,\lyxmathsym{\textemdash}V\}$ over $M$
is like a vector field with undetermined sign (i.e. determined up
to a factor of $\pm1$) at each point $p$ of $M$; this is a \textit{smooth}
assignment to each $p\in M$ of a zero-dimensional sub-bundle of the
tangent bundle. Analogously, we can obtain a one-form field $(q,-q)$
up to sign, if we use $g$ to lower indicies. Then $w=q\otimes q$
is a well-defined $2$-covariant symmetric tensor~\cite{Lerner}.
For any smooth manifold $M$ the existence of a Lorentzian metric
is equivalent to the existence of a global, smooth non-vanishing line
element field~\cite{Hall,Markus - Line-Element Fields and Lorentz Structures on Differentiable Manifolds,Steenrod}
on $M$. The set of all non-vanishing line element fields $\{V,-V\}$
on $M$ is denoted by $\mathcal{L}(M)$. And henceforth, let $\mathfrak{F}(M)$
be the set of all smooth real-valued functions on $M$.
\begin{rem}
The existence of a global non-vanishing line element field is equivalent
to the existence of a one-dimensional distribution. Also, the latter
condition is equivalent to the manifold $M$ admitting a $C^{\infty}$
Lorentzian metric~\cite{Hall}.
\end{rem}

~

Before we prove Proposition~\ref{Proposition Transformation-Prescription-1}
below, let's recall that the Lorentzian metric $g$ induces the musical
isomorphisms between the tangent bundle $TM$ and the cotangent bundle
$T^{*}M$: Recall that \textit{flat} $\flat$ is the vector bundle
isomorphism $\flat\colon TM\rightarrow T^{*}M$ induced by the isomorphism
of the fibers $\flat\colon T_{p}M\rightarrow T_{p}^{*}M$, given by
$v\mapsto v^{\flat}$, where $\flat(v)(w)=v^{\flat}(w)=g(v,w)$ $\forall\,v,w\in T_{p}M$
for all $p\in M$.

~\\ We thus obtain for the vector $v=v^{j}e_{j}$ and the associated
$1$-form $v^{\flat}=v_{k}e^{k}$ the expression for the musical isomorphism
in local coordinates:~\\ $v^{\flat}(w)=v_{k}e^{k}(w^{i}e_{i})=v_{k}w^{i}\delta_{i}^{k}=v_{i}w^{i}=g(v^{j}e_{j},w^{i}e_{i})=g_{ij}v^{j}w^{i}$,
where $w=w^{i}e_{i}$ is an arbitrary vector. This calculation yields
$v_{i}=g_{ij}v^{j}\Longrightarrow v^{\flat}=v_{i}e^{i}=g_{ij}v^{j}e^{i}$.\textbf{~}\\{}

\begin{prop*}[Transformation Prescription]
\label{Proposition Transformation-Prescription}
Let $(M,g)$ be an $n$-dimensional (not necessarily time-orientable)
Lorentzian manifold. Then we obtain a signature-type changing metric
$\tilde{g}$ via the Transformation Prescription $\tilde{g}=g+fV^{\flat}\otimes V^{\flat}$,
where $f\colon M\longrightarrow\mathbb{R}$ is a smooth transformation
function and $V$ is one of the unordered pair $\{V,\lyxmathsym{\textemdash}V\}$
of a global smooth non-vanishing  line element field.
\end{prop*}
\begin{proof}
Since a smooth non-vanishing line element field exists for any Lorentzian
manifold and is defined as an assignment of a non-ordered pair of
equal and opposite vectors $\{V,-V\}$ at each point $p\in M$, the
existence of a smooth global non-vanishing line element field is always
ensured (even in the case of a non time-orientable Lorentzian manifold).
Hawking and Ellis~\cite{Hawking + Ellis - The large scale structure of spacetime},
p. 40, elucidate how to construct a smooth global non-vanishing \textit{timelike}
line element field from a smooth global non-vanishing line element
field. And since $V$ is a line element field that is non-zero at
every point in $M$, it can be normalized by $g_{\mu\nu}V^{\mu}V^{\nu}=-1$.\footnote{Here the term ``timelike'' refers to the Lorentzian metric $g$.}

\textbf{~}\\ Now consider an arbitrary $C^{\infty}$ function $f\colon M\rightarrow\mathbb{R}$,
and define tensor fields on $M$ of the form $\tilde{g}\coloneqq g+f(V^{\flat}\otimes V^{\flat})$,
where $\flat$ denotes the musical isomorphism. Locally there exist
vector fields $E_{i}$ that form an orthonormal frame, such that we
can express the metric in terms of coordinates:

\[
\tilde{g}_{\mu\nu}=g_{\mu\nu}+f(V_{\mu}V_{\nu})=g_{\mu\nu}+f(g_{\mu\alpha}V^{\alpha}E^{\mu}\otimes g_{\nu\beta}V^{\beta}E^{\nu})
\]

\[
=g_{\mu\nu}+f(g_{\mu\alpha}V^{\alpha}g_{\nu\beta}V^{\beta}(E^{\mu}\otimes E^{\nu}))=g_{\mu\nu}+f(g_{\mu\alpha}V^{\alpha}g_{\nu\beta}V^{\beta}).
\]
\textbf{~}\\ Moreover, locally there exist vector fields $E_{j}$
such that $\{V,E_{2},\ldots,E_{n}\}$ is a Lorentzian frame field
relative to $g$. Then

\[
\tilde{g}(E_{i},E_{j})=g(E_{i},E_{j})+f(V^{\flat}\otimes V^{\flat})(E_{i},E_{j})=\delta_{ij}+f(\underset{0}{\underbrace{g(V,E_{i})}}\cdot\underset{0}{\underbrace{g(V,E_{j})}})=\delta_{ij},
\]

\[
\tilde{g}(V,E_{j})=g(V,E_{j})+f(\underset{-1}{\underbrace{g(V,V)}}\cdot\underset{0}{\underbrace{g(V,E_{j})}})=0,
\]

\[
\tilde{g}(V,V)=g(V,V)+f(\underset{-1}{\underbrace{g(V,V)}}\cdot\underset{-1}{\underbrace{g(V,V)}})=f-1.
\]
Consequently, because of $0>\tilde{g}(V,V)=f-1\Leftrightarrow1>f$,
$\tilde{g}$ is a Lorentzian metric on $M$ in the region with $1>f(p)$.
Analogously, in the region with $1<f(p)$, we have that $\tilde{g}$
is a Riemannian metric, and for $f(p)=1$ the metric $\tilde{g}$
is degenerate. 
\end{proof}
~ 

If $1$ is a regular value of $f\colon M\rightarrow\mathbb{R}$, then
$\mathfrak{\mathcal{H}\coloneqq}f^{-1}(1)=\{p\in M\colon f(p)=1\}$
is a hypersurface in $M$. This is to say, $\mathcal{H}$ is a submanifold
of dimension $n-1$ in $M$, constituting the locus where the signature
change occurs. Moreover, for every $q\in\mathcal{H}$, the tangent
space $T_{q}\mathcal{H}$ is the kernel $\ker df_{q}=T_{q}(f^{-1}(1))$
of the map $df_{q}\colon T_{q}M\longrightarrow T_{1}\mathbb{R}$.
According to this, $(M,\tilde{g})$ represents a signature-type changing
manifold with the locus of signature change at $\mathcal{H}$.

~\\{}

\subsection{Representation of $\tilde{g}$\label{subsec:Representation-of (g,V,f)}}

The representation of $\tilde{g}$ as $\tilde{g}=g+fV^{\flat}\otimes V^{\flat}$,
as introduced in the \textit{Transformation Prescription} (Proposition~\ref{Proposition Transformation-Prescription-1})
above, is ambiguous. More precisely, different triples $(g,V,f)$
can yield the same metric $\tilde{g}$. To see this, notice that for
$(M,\tilde{g})$ with $\dim(M)=n$, the signature-type changing metric
$\tilde{g}$ and the Lorentzian metric $g$ are determined pointwise
by $\frac{n(n+1)}{2}$ metric coefficients. Aside from that, for the
former we have the $n$ components of $V$ and the value of $f$ (besides
$1$, because of $g(V,V)=-1$) at each point in $M$. 

\textbf{~}

In the following, we demonstrate that one can choose either $V$ or
$f$ arbitrarily, with the fixed condition that $f(q)=1$ $\forall q\in\mathcal{H}$
on the hypersurface, but not both independently of each other.
\begin{prop}
\label{prop:Choice of (V,g,f)}Given a signature-type changing metric
$\tilde{g}$, then in particular, the line element field $V$ can
be chosen arbitrarily (subject to the condition of being timelike
in $M_{L}$ with respect to $\tilde{g}$). This choice subsequently
allows for the determination of the Lorentzian metric $g$ and a $C^{\infty}$
function $f$, with the fixed condition $f(q)=1$ $\forall q\in\mathcal{H}$.
Then the triple $(g,V,f)$ constitutes the representation of $\tilde{g}$
through $\tilde{g}=g+fV^{\flat}\otimes V^{\flat}$, as introduced
in the Transformation Prescription (Proposition~\ref{Proposition Transformation-Prescription-1}).
\end{prop}

\begin{proof}
Let $\psi\colon M_{R}\cup M_{L}\longrightarrow\mathbb{R}$ be chosen
such that $\tilde{V}\coloneqq\psi V$ is normalized with respect to
$\tilde{g}$: that is, $\tilde{g}(\tilde{V},\tilde{V})=1$ and $\tilde{g}(\tilde{V},\tilde{V})=-1$
in $M_{R}$ and $M_{L}$, respectively.~\\ In the following, only
the relationships in $M_{L}$ will be examined, the argumentation
in the Riemannian sector $M_{R}$ is carried out analogously with
corresponding changes in sign. Consider the normalized line element
field $\tilde{V}$, with ~\\ 
\[
-1=\tilde{g}(\tilde{V},\tilde{V})=g(\tilde{V},\tilde{V})+f[g(V,\tilde{V})]^{2}
\]

\[
=\psi^{2}g(V,V)+f\psi^{2}[g(V,V)]^{2}=\psi^{2}\cdot(-1+f)
\]

\begin{equation}
\Longrightarrow f-1=-\frac{1}{\psi^{2}}.\label{eq:determined f}
\end{equation}
Then, we extend $\tilde{V}$ to a Lorentzian basis $\{\tilde{V},\tilde{E}_{1},\ldots,\tilde{E}_{n-1}\}$
relative to $\tilde{g}$, and we get

\[
0=\tilde{g}(\tilde{V},\tilde{E}_{i})=g(\tilde{V},\tilde{E}_{i})+fg(V,\tilde{V})g(V,\tilde{E}_{i})
\]

\[
=\psi[g(V,\tilde{E}_{i})+fg(V,V)g(V,\tilde{E}_{i})]=\psi\cdot(1-f)g(V,\tilde{E}_{i})
\]

\begin{equation}
\Longrightarrow g(V,\tilde{E}_{i})=0,\label{eq: determinet g(V,E)}
\end{equation}
because of $\psi\cdot(1-f)\neq0$ on $M_{R}\cup M_{L}$. Furthermore, 

\[
\delta_{ij}=\tilde{g}(\tilde{E}_{i},\tilde{E}_{j})=g(\tilde{E}_{i},\tilde{E}_{j})+f\underset{0}{\underbrace{g(V,\tilde{E}_{i})}}\underset{0}{\underbrace{g(V,\tilde{E}_{j})}}
\]

\begin{equation}
\Longrightarrow g(\tilde{E}_{i},\tilde{E}_{j})=\delta_{ij}.\label{eq:Determined g(E,E)}
\end{equation}
Since $\{V,\tilde{E}_{1},\ldots,\tilde{E}_{n-1}\}$ is also a basis,
and in addition $g(V,V)=-1$ holds, the metric $g$ is uniquely determined.
Moreover, for the function $\psi\colon M_{R}\cup M_{L}\longrightarrow\mathbb{R}$
we have established the relation $\tilde{g}(V,V)=\frac{1}{\psi^{2}}\tilde{g}(\tilde{V},\tilde{V})=-\frac{1}{\psi^{2}}$,
and therefore, the $C^{\infty}$ function $f$, based on Equation~\ref{eq:determined f},
is also uniquely determined by $f=1-\frac{1}{\psi^{2}}$.~\\ Note
that a change in length defined by $\phi\colon M_{R}\cup M_{L}\longrightarrow\mathbb{R}$,
$V\mapsto\phi V$, entails the change $f\mapsto1+\phi^{2}\cdot(f-1)$. 
\end{proof}
In light of Proposition~\ref{prop:Choice of (V,g,f)}, we understand
that, due to arbitrary rescaling, there are no distinguished values
for $f$. Moreover, according to the above proposition, the triples
$(g,V,f)$ form equivalence classes, where all triples within an equivalence
class yield the same metric $\tilde{g}$. If one perturbs a triple,
especially at $V$ but not at $g$ and $f$, the new triple belongs
to a different equivalence class and thus yields a different $\tilde{g}$.
However, within the new equivalence class, there is also a triple
with the original $V$ that similarly yields the ``new'' $\tilde{g}$. 
On the other hand, one can \textquoteleft simultaneously\textquoteright{}
perturb $V$, $g$, and $f$ in such a way that the original equivalence
class is maintained, and hence, the original $\tilde{g}$ is preserved.
This insight suggests the following proposition.

~
\begin{prop}
\label{prop:Equivalence-classes-(V,g,f)}Let $X$ be the set of all
triples $(g,V,f)$ with $g\in Lor(M)$, $f\in\mathfrak{F}(M)$, $V\in\mathcal{L}(M)$,
where $Lor(M)$ denotes the set of all Lorentzian metrics on $M$,
$\mathfrak{F}(M)$ the set of all smooth real-valued functions on
$M$, and $\mathcal{L}(M)$ the set of all non-vanishing line element
fields $\{V,-V\}$ on $M$. The equivalence relation $\sim$ on $X$
is defined by: $(g,V,f)\sim(\bar{g},\bar{V},\bar{f})$ if and only
if $\tilde{g}=g+fV^{\flat}\otimes V^{\flat}=\bar{g}+\bar{f}\bar{V}^{\flat}\otimes\bar{V}^{\flat}$.
Then the partition $X$ of the set of all triples $(g,V,f)$ is given
by $X/\sim\,=\{[\tilde{g}]_{\sim}=[(g,V,f)]_{\tilde{g}}\colon(g,V,f)\in X\}$,
where $\tilde{g}$ is to be interpreted in such a way that it can
be regarded as a representative of the equivalence class of triples
corresponding to $\tilde{g}$.
\end{prop}

\begin{proof}
First recall that a relation $\sim$ on  $X$ is called an equivalence
relation if it is reflexive, symmetric, and transitive. Moreover,
a partition of the set $X$ is then defined as a collection of all
disjoint non-empty subsets $X_{i}$ of $X$, where $i\in I$ ($I$
is the index set), such that

~

$X_{i}\neq0$ $\forall i\in I$, 

$X_{i}\cap X_{j}=\emptyset$, when $i\neq j$, 

$\bigcup_{i\in I}X_{i}=X$.

\textbf{~}\\ It is straightforward to demonstrate that the relation
$\sim$ satisfies all three conditions for an equivalence relation.
Specifically, two triples $(g,V,f)\sim(\bar{g},\bar{V},\bar{f})$
are equivalent if and only if $\tilde{g}=g+fV^{\flat}\otimes V^{\flat}=\bar{g}+\bar{f}\bar{V}^{\flat}\otimes\bar{V}^{\flat}$.
Given the relation $\sim$, we can define the equivalence class $[\tilde{g}]_{\sim}=[(g,V,f)]_{\tilde{g}}=\{(g,V,f)\in Lor(M)\times\mathcal{L}(M)\times\mathfrak{F}(M)\colon\tilde{g}=g+fV^{\flat}\otimes V^{\flat}\}$,
where $\tilde{g}$ can be viewed as a class representative of the
equivalence class of triples corresponding to $\tilde{g}$. And we
can establish $X/\sim \ =\{[\tilde{g}]_{\sim}=[(g,V,f)]_{\tilde{g}}\colon(g,V,f)\in X\}$
$X/\sim\,=\{[\tilde{g}]_{\sim}\,=[(g,V,f)]_{\tilde{g}}\colon(g,V,f)\in X\}$,
which is a pairwise disjoint partition of $X$. Note that the set
of class representatives $\tilde{g}$ is a subset of $X$ which contains
exactly one element from each equivalence class $[\tilde{g}]_{\sim}=[(g,V,f)]_{\tilde{g}}$,
this is the set of all signature-type changing metrics $\tilde{g}$
on $M$. 
\end{proof}
\begin{cor}
There is a bijection between the partition of the set of all triples
$(g,V,f)$ and the set of all signature-type changing metrics $\tilde{g}$
on $M$.
\end{cor}

\section{Transformation theorem}

Note that in the \textit{Transformation Prescription} (Proposition~\textit{\ref{Proposition Transformation-Prescription-1}})
the locus of signature-change is not necessarily an embedded hypersurface
in $M$. Recall that this is only the case if $1$ is a regular value
of $f\colon M\longrightarrow\mathbb{R}$, and then $\mathfrak{\mathcal{H}\coloneqq}f^{-1}(1)$
is a smoothly embedded hypersurface in $M$. Also, the signature-type
changing manifold $(M,\tilde{g})$ has a spacelike hypersurface if
and only if the radical $\textrm{Rad}_{q}$ intersects $T_{q}\mathcal{H}$
transversally for all $q\in\mathcal{H}$.

\subsection{Local Transformation Theorem}

Our main result, Theorem~\ref{Transformation-Theorem-(local)-1},
can now be proved:

\begin{thm*}[Local Transformation Theorem, transverse radical]
\label{Transformation-Theorem-(local)}
For every $q\in M$ there exists a neighborhood $U(q)$, such that
the metric $\tilde{g}$ associated with a signature-type changing
manifold $(M,\tilde{g})$ is a transverse, type-changing metric with
a transverse radical if and only if $\tilde{g}$ is locally obtained
from a Lorentzian metric $g$ via the Transformation Prescription
$\tilde{g}=g+fV^{\flat}\otimes V^{\flat}$, where, for all $q\in \mathcal H \coloneqq f^{-1}(1)=\{p\in M\colon f(p)=1\}$,
\[
df(q)\neq0 \quad \text{and} \quad (df(V))(q)\neq0.
\]\end{thm*}

\begin{proof}
In the subsequent proof, we only consider the scenario where $q\in\mathcal{H}$.
If $q$ is within the Lorentzian sector $M_{L}$, then $U(q)$ can
be chosen to be sufficiently small, ensuring that $U(q)$ is entirely
contained within $M_{L}$. Consequently, in this scenario, the theorem's
assertion becomes trivial, as $\tilde{g}$ already represents a Lorentzian
metric there, and thus, $f=0$ satisfies all the stated conditions.
Similarly, we can select a neighborhood in the Riemannian sector $M_{R}$
where, for instance, $f=2$ is trivially applicable.

\textbf{~}\\ To begin, consider that, according to Proposition~\ref{prop:Choice of (V,g,f)}
and Proposition~\ref{prop:Equivalence-classes-(V,g,f)} any triple
$(g,V,f)$, where $g\in Lor(M)$, $V\in\mathcal{L}(M)$ with $g(V,V)=-1$
and $f\in\mathfrak{F}(M)$, yields a signature-type changing metric
$\tilde{g}=g+fV^{\flat}\otimes V^{\flat}=g+fg(V,\centerdot)g(V,\centerdot)$,
which is defined over the entire manifold $M$ (see\textit{ }Proposition~\ref{Proposition Transformation-Prescription-1}).
Conversely, if we have a signature-type changing metric $\tilde{g}$
we can always single out the associated triple $(g,V,f)$ belonging
to the equivalence class of $[\tilde{g}]$, such that $\tilde{g}$
locally takes the form $\tilde{g}=g+fV^{\flat}\otimes V^{\flat}$.
In either case, we can initiate the proof by assuming that locally
$\tilde{g}=g+fV^{\flat}\otimes V^{\flat}$ is given, and normalized
with $g(V,V)=-1$.

\textbf{~}\\ In order to simplify the problem as much as possible
we adopt co-moving coordinates (refer to~\cite{Fabbri,Rindler,Schutz}
for more details), that is $g_{00}(V^{0})^{2}=g(V,V)=-1$ and $V^{i}=0$
for $i\neq0$. Then the metric $\tilde{g}$ relative to these coordinates
is given by\textbf{ }

\begin{equation}
\tilde{g}_{\mu\nu}=g_{\mu\nu}+fg_{\mu\alpha}V^{\alpha}g_{\nu\beta}V^{\beta}=g_{\mu\nu}+fg_{\mu0}g_{\nu0}(V^{0})^{2}=g_{\mu\nu}-f\frac{g_{\mu0}g_{\nu0}}{g_{00}}.\label{eq: g-tilde in comoving coordinates}
\end{equation}

\textbf{~}\\ The last equality follows from the aforementioned condition
\[
g_{00}(V^{0})^{2}=g(V,V)=-1\Longleftrightarrow(V^{0})^{2}=-\frac{1}{g_{00}}.
\]
\textbf{~}\\ According to Equation~\ref{eq: g-tilde in comoving coordinates}
the components of the metric in co-moving coordinates are determined
by

~

$\tilde{g}_{00}=g_{00}-f\frac{(g_{00})^{2}}{g_{00}}=(1-f)g_{00}$,

$\tilde{g}_{ij}=g_{ij}-f\frac{g_{i0}g_{j0}}{g_{00}}$,

$\tilde{g}_{0i}=g_{0i}-fg_{i0}=(1-f)g_{0i}$.

\textbf{~}\\ And the associated matrix representation of $\tilde{g}$
is given by 

\textbf{~}\\ $[\tilde{g}_{\mu\nu}]=\left (\begin{array}{@{}c|ccc@{}}      (1-f)g_{00} & (1-f)g_{01} & \cdots & (1-f)g_{0n-1}   \\\hline (1-f)g_{10} & \tilde{g}_{11} & \cdots & \tilde{g}_{1n-1}  \\     \vdots & \vdots &  &\vdots  \\ (1-f)g_{n-10} & \tilde{g}_{n-11} & \cdots &\tilde{g}_{n-1n-1}     \end{array}\right)$ \textbf{~}

\textbf{~}\\{}

\ \ $= \left (\begin{array}{@{}c|ccc@{}}      (1-f)g_{00} & (1-f)g_{01} & \cdots & (1-f)g_{0n-1}  \\\hline     (1-f)g_{10} &  &  &   \\     \vdots &  & \tilde{G} &  \\     (1-f)g_{n-10} &  &  &      \end{array}\right).$ 

\textbf{~}

\textbf{~}\\ Then take the determinant $\det([\tilde{g}_{\mu\nu}])=(1-f)\det(G_{\mu\nu})$,
where 

\textbf{~}\\  $G_{\mu\nu}= \left (\begin{array}{@{}c|ccc@{}}      g_{00} & g_{01} & \cdots & g_{0n-1}  \\\hline     (1-f)g_{10} &  &  &   \\     \vdots &  & \tilde{G} &  \\     (1-f)g_{n-10} &  &  &      \end{array}\right)$.

\textbf{~}

\textbf{~}\\ Now consider $d(\det([\tilde{g}_{\mu\nu}]))=d[(1-f)\det(G_{\mu\nu})]=(1-f)d(\det(G_{\mu\nu}))-df\det(G_{\mu\nu})$.
Since $q\in\mathfrak{\mathcal{H}\coloneqq}f^{-1}(1)=\{p\in M\colon f(p)=1\}$
is a regular point for $f$, the term $1-f=0$ is zero on the hypersurface
$\mathcal{H}$. Hence, on $\mathcal{H}$ we are left with $d(\det([\tilde{g}_{\mu\nu}])_{q})=-df\cdot\det(G_{\mu\nu})$.
On $\mathcal{H}$ we have $f\equiv1$, and therefore it remains to
show that on $\mathcal{H}$ the following is true:

\textbf{~}\\ $0\neq\det(G_{\mu\nu})=\det\left(\begin{array}{cccc}
g_{00} & g_{01} & \cdots & g_{0n-1}\\
0 & g_{11}-\frac{g_{01}g_{01}}{g_{00}} & \cdots & g_{1n-1}-\frac{g_{01}g_{0n-1}}{g_{00}}\\
\vdots & \vdots & \ddots & \vdots\\
0 & g_{n-11}-\frac{g_{0n-1}g_{01}}{g_{00}} & \cdots & g_{n-1n-1}-\frac{g_{0n-1}g_{0n-1}}{g_{00}}
\end{array}\right)$.

\textbf{~}

\textbf{~}\\ Notice that \(\det(G_{\mu\nu})\) is the determinant of a block matrix with the block \([g_{00}] \neq 0\) being invertible on \(\mathcal{H}\) (note that \(g_{00}\) is also non-zero outside of \(\mathcal{H}\), simply as a consequence of \(g_{00}(V^{0})^{2} = g(V, V) = -1\), while, of course, \(\tilde{g}_{00}\) is zero exactly on \(\mathcal{H}\)), hence

\[
\det(G_{\mu\nu})=\det([g_{00}])\cdot\det([g_{ij}-\frac{g_{i0}g_{j0}}{g_{00}}])=g_{00}\cdot\det(\tilde{G}).
\]
Because of the condition $g_{00}(V^{0})^{2}=-1$ we have 
\[
g_{ij}-\frac{g_{i0}g_{j0}}{g_{00}}=g_{ij}+g_{0i}g_{0j}(V^{0})^{2}=:h_{ij},
\]
\textbf{~}\\ which are the metric coefficients of the degenerate
metric $h=g+fV^{\flat}\otimes V^{\flat}$, i.e. the coefficients of
$[\tilde{g}_{\mu\nu}]$ for $f=1$. More precisely, $h_{\mu\nu}=g_{\mu\nu}+g_{\mu\alpha}V^{\alpha}g_{\nu\beta}V^{\beta}=g_{\mu\nu}+g_{\mu0}g_{\nu0}(V^{0})^{2}$.
The metric $h$, restricted to the $g$-orthogonal complement $V^{\bot}$ of $V$ in each tangent
space, results in the non-degenerate \textquotedbl spatial metric\textquotedbl{}
with metric coefficients $h_{ij}$. 

~

\noindent
Note that the ``spatial metric'' is not necessarily the metric of a spacelike hypersurface, but in the case that the orthogonal complements of \( V \) are not integrable, rather just a bilinear form in each orthogonal complement of \( V \). As a matter of fact, the integrability condition states that the ``mixed''  metric coefficients can all be transformed to zero using an appropriate coordinate transformation (within the class of ``co-moving coordinates''). This condition can also coordinate-independently be expressed in terms of the rotation of the vector field \( V \). It ensures that (at least locally) all orthogonal complements of \( V \) fit together as tangent spaces of a family of spacelike hypersurfaces (as integral manifolds). The metric \( h \) is always (and everywhere) non-degenerate, since it lives in the orthogonal complement of the timelike vector \( V \) of the Lorentz manifold \( (M,g) \) and is simply the restriction of \( g \), and thus positive definite in that region. In particular, this has nothing to do with the signature change of the metric \( \tilde{g} \), and especially nothing to do with the radical.

~

Finally, we may show that from
$\det(G_{\mu\nu})=g_{00}\cdot\det(\tilde{G})=g_{00}\cdot\det([h_{ij}])$
on $\mathcal{H}$, and the fact that the ``spatial metric'' $h$
is non-degenerate, follows that $\det(G_{\mu\nu})\neq0$ on $\mathcal{H}$.
Consequently, from $d(\det([\tilde{g}_{\mu\nu}]))=-df\cdot\det(G_{\mu\nu})$
on $\mathcal{H}$ we get the biconditional statement 
\[
d(\det([\tilde{g}_{\mu\nu}]))(q)\neq0\;\forall q\in\mathcal{H}\Longleftrightarrow df(q)\neq0\;\forall q\in\mathcal{H}.
\]

\textbf{~}\\ Observe that for every $q\in\mathcal{H}$, $T_{q}\mathcal{H}$
is the kernel of the map $df_{q}\colon T_{q}M\longrightarrow T_{1}\mathbb{R}$.
Therefore the condition $V(f)=(df)(V)\neq0$, $\forall q\in\mathcal{H}$
ensures that $V\notin T_{q}\mathcal{H}$ and thus $V$ is not tangent
to $\mathcal{H}$. By additionally requiring that $(V(f))(q)=((df)(V))(q)\neq0$,
the radical is guaranteed to be transverse. 
\end{proof}
\begin{rem}
Remember that the triples $(g,V,f)$ form equivalence classes (see
Proposition~\ref{prop:Equivalence-classes-(V,g,f)}), where all triples
within an equivalence class yield the same metric $\tilde{g}$. Hence,
by picking an arbitrary triple $(g,V,f)$ in co-moving coordinates,
we have shown that the relation \textit{$d(\det([\tilde{g}_{\mu\nu}]))\neq0\;\forall q\in\mathcal{H}\Longleftrightarrow df\neq0\;\forall q\in\mathcal{H}$
}holds independently of a choice of coordinates. This result is coordinate-independent
due to the nature of $f$ as a scalar, even though the determinant
depends on coordinates. However, whether the differential of the determinant
on the hypersurface is zero or not is independent of the choice of
coordinates. This follows from the nonsingularity of the Jacobian
matrix associated with a coordinate transformation, in conjunction
with the multiplicativity of the determinant in matrix multiplication,
and the vanishing of the determinant on the hypersurface.

\textbf{~}

The equivalence $d(\det([\tilde{g}_{\mu\nu}]))\neq0\Longleftrightarrow df\neq0$
is, in this sense, a local statement, as it only holds on the hypersurface.
On the other hand, the statement's coordinate independence on the
hypersurface $\mathcal{H}$ implies that it remains unaffected by
the choice of coordinates, ensuring its validity across the entire
hypersurface. Thus, it possesses a global character in this regard---this aspect could be called \textit{$\mathcal{H}$-global}, see Definition in~\cite{Rieger-Embedding}. 
\end{rem}

Since for every $q\in\mathcal{H}$, the tangent space $T_{q}\mathcal{H}$
is the kernel of the map $df_{q}\colon T_{q}M\longrightarrow T_{1}\mathbb{R}$,
the condition $V(f_{q})=df_{q}(V)\neq0$, $\forall q\in\mathcal{H}$
ensures that $V\notin T_{q}\mathcal{H}$. Hence, $V$ is not tangent
to $\mathcal{H}$. This guarantees that the radical in $(M,\tilde{g})$
is transverse. If we are ready to relax our constraints and do not
impose this restriction on $V$, then we get a slightly modified version
of the Transformation Theorem, such that the following corollary holds.

\begin{cor}
\label{Corollary:General Transformation Theorem}For every $q\in M$
there exists a neighborhood $U(q)$, such that the metric $\tilde{g}$
associated with a signature-type changing manifold $(M,\tilde{g})$
is a transverse, type-changing metric with a transverse radical only
if $\tilde{g}$ is locally obtained from a Lorentzian metric $g$
via the Transformation Prescription $\tilde{g}=g+fV^{\flat}\otimes V^{\flat}$,
as introduced in Proposition~\ref{Proposition Transformation-Prescription-1},
where $df(q)\neq0$ for every $q\in\mathfrak{\mathcal{H}\coloneqq}f^{-1}(1)=\{p\in M\colon f(p)=1\}$.
\end{cor}

Contrariwise, if the additional constraint $(V(f))(q)=((df)(V))(q)=0$
for every $q\in\mathfrak{\mathcal{H}\coloneqq}f^{-1}(1)=\{p\in M\colon f(p)=1\}$
is imposed on the Transformation Prescription (Proposition~\ref{Proposition Transformation-Prescription-1}),
then we get an alternative version of the Transformation Theorem:

\begin{thm}[Local Transformation Theorem, tangent radical]
\label{Transformation-Theorem-(local)-tangent}
For every $q\in M$ there exists a neighborhood $U(q)$, such that
the metric $\tilde{g}$ associated with a signature-type changing
manifold $(M,\tilde{g})$ is a type-changing metric with a tangent
radical if and only if $\tilde{g}$ is locally obtained from a Lorentzian
metric $g$ via the Transformation Prescription $\tilde{g}=g+fV^{\flat}\otimes V^{\flat}$,
where, for all $q\in \mathcal H \coloneqq f^{-1}(1)=\{p\in M\colon f(p)=1\}$,
\[
df(q)\neq0 \quad \text{and} \quad (df(V))(q)\neq0.
\]
\end{thm}

\begin{proof}
The alternative condition $V(f)=(df)(V)=0$, $\forall q\in\mathcal{H}$
ensures that $V\in T_{q}\mathcal{H}$, and thus $V$ is tangent to
$\mathcal{H}$. 
\end{proof}
\begin{example}
\label{exa:Example-18}Consider on $\mathbb{R}^{2}$ the metric $ds^{2}=x(dt)^{2}+(dx)^{2}$
(\cite{Kriele + Martin - Black Holes ...}, page 9). This is a signature-type
changing metric with

~

$\triangle\coloneqq\det([g_{ij}])=x$, 

$d\triangle=\frac{\partial x}{\partial t}dt+\frac{\partial x}{\partial x}dx=dx$,

$\mathcal{H}=\{q\in M\mid x(q)=0\}$.

\textbf{~}\\ The differential is $d\triangle=dx\neq0$ on $\mathcal{H}$,
with $\mathcal{H}$ being a smoothly embedded hypersurface. The $1$-dimensional
radical is given by $\textrm{Rad}_{q}=\textrm{span}\{\frac{\partial}{\partial t}\}$
for $q\in\mathcal{H}$, and it is tangent with respect to $\mathcal{H}$.

~

\begin{figure}[H]
\centering{}\includegraphics[scale=0.66]{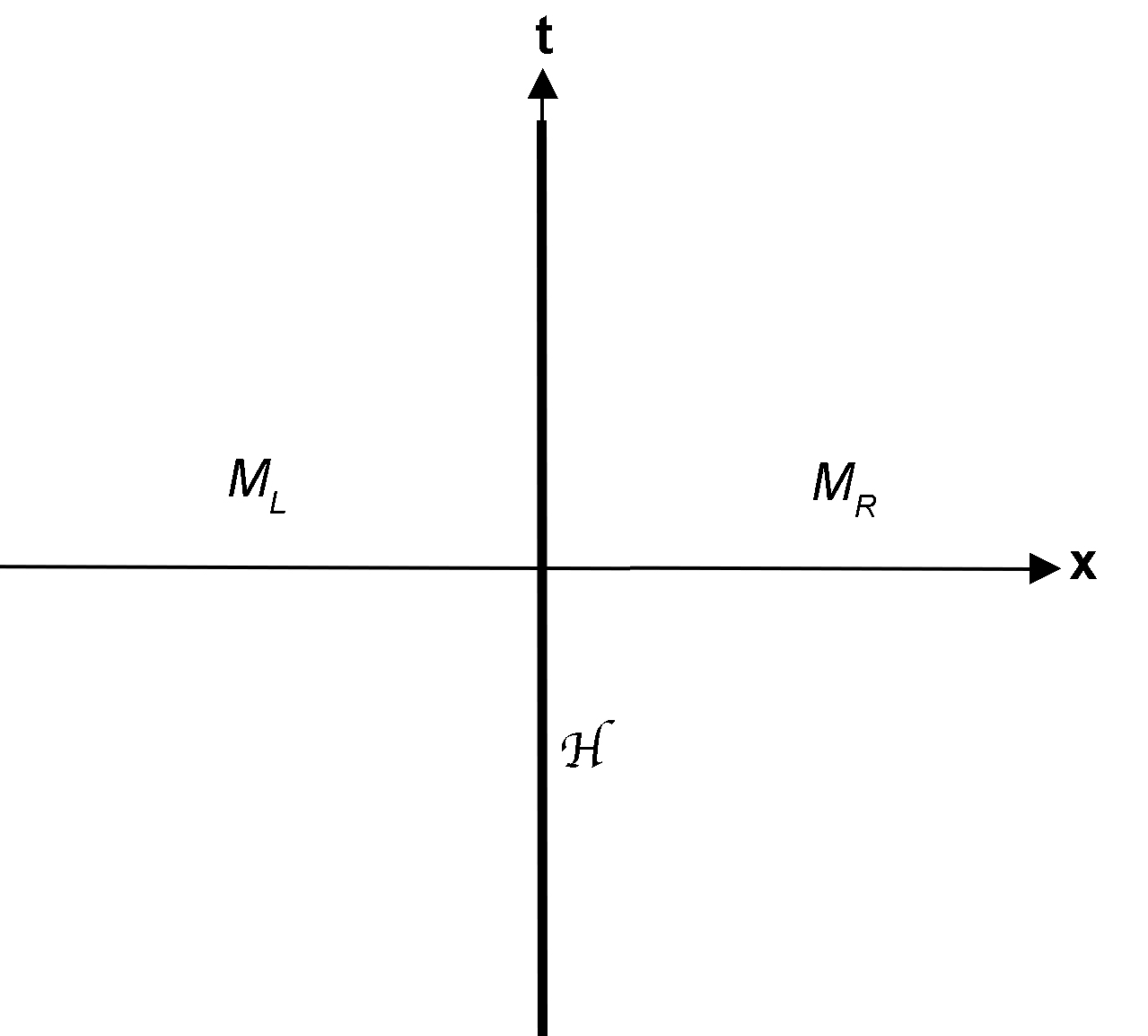}\caption{{\small{}The radical is tangent with respect to the hypersurface $\mathcal{H}$.
The latter one is located between the Lorentzian region $M_{L}$ and
the Riemannian region $M_{R}$.}}
\end{figure}

\textbf{~}\\ Here the transverse, signature-type changing metric
$ds^{2}=x(dt)^{2}+(dx)^{2}$ is \textit{given}. So we need to \textit{find}
a suitable Lorentzian metric $g$ as well as a global smooth function
$f\colon M\longrightarrow\mathbb{R}$ and a non-vanishing line element
field $V$, both with the properties mentioned in~\ref{Corollary:General Transformation Theorem}.

\textbf{~}\\ By an educated guess, we start with $g=-(dt)^{2}+(dx)^{2}$,
from which follows $x(dt)^{2}=-(dt)^{2}+fV^{\flat}\otimes V^{\flat}$.
As non-vanishing line element field we pick the one related to $V=\frac{\partial}{\partial t}$,
with $V^{\flat}=g(V,\centerdot)=g(\frac{\partial}{\partial t},\centerdot)=-(dt\otimes dt)(\frac{\partial}{\partial t},\centerdot)=-dt$.
Hence, $x(dt)^{2}=-(dt)^{2}+f\cdot(dt)^{2}\Longrightarrow f(t,x)=1+x$.\textbf{~}\\ Consequently
this yields $\tilde{g}=g+(1+x)(dt)^{2}=-(dt)^{2}+(dx)^{2}+(1+x)(dt)^{2}$
with $f(t,x)=1+x$. Then on $\mathcal{H}$ for $x=0\Longleftrightarrow f(t,0)=1$
we have $\tilde{g}(V,X)=(dx)^{2}(V,X)=0\;\forall X$, and this means
$V\in \textrm{Rad}_{p(t,0)}\;\forall t$. And since $\textrm{Rad}_{q(t,0)}=\textrm{span}\{V\}=\textrm{span}(\{\frac{\partial}{\partial t}\})$
the radical is tangent with respect to $\mathcal{H}$.
\end{example}

\subsection{Global Transformation Theorem}

By imposing an extra constraint on the local variant of the Transformation
Theorem~\ref{Transformation-Theorem-(local)-1}, we can establish
the validity of the global version of the Transformation Theorem:
The Riemannian sector with boundary $M_{R}\cup\mathcal{H}$ is required
to possess a smoothly defined non-vanishing line element field that
is transverse to the boundary $\mathcal{H}$. The hypersurface $\mathcal{H}$
can be viewed as common connected boundary of the Riemannian region
$M_{R}$ and the Lorentzian region $M_{L}$. A vector field $V$ is
transverse to the boundary $\mathcal{H}$ if, at each point $p\in\mathcal{H}$,
$V(q)\notin T_{q}\mathcal{H}$ (i.e. $V(p)$ is not tangent to $\mathcal{H}$).
This typically means $V$ has a component normal to $\mathcal{H}$
that is non-zero. Recall that a manifold with boundary $M_{R}\cup\mathcal{H}$
is a topological space in which each point $p\in M_{R}\cup\mathcal{H}$
has a neighborhood homeomorphic to either $\mathbb{R}^{n}$ or the
upper half-space $\mathbb{H}^{n}=\{(x_{1},x_{2},\ldots,x_{n})\}\in\mathbb{R}^{n}:x_{n}\geq0\}$
for some $n$. In the subsequent discussion, we explore the means
by which this supplementary constraint can be guaranteed.

\subsubsection{Non-vanishing vector field in $M_{R}\cup\mathcal{H}$ transverse
to $\mathcal{H}$}

There are certain results and theorems in differential topology that
provide conditions under which a non-vanishing vector field exists
on a manifold~\cite{Hall,Hirsch,Spivak}. For more general manifolds,
the existence of a non-vanishing vector field is related to the topology
of the manifold:

\smallskip{}

(i) If $M_{R}\cup\mathcal{H}$ is a noncompact connected manifold
with boundary, then it admits a nowhere vanishing vector field.

\vspace{0.01\baselineskip}

(ii) If $M_{R}\cup\mathcal{H}$ is compact and connected with boundary,
then $M_{R}\cup\mathcal{H}$ admits a nowhere vanishing vector field
if $\chi(M_{R}\cup\mathcal{H})=0$, where $\chi$ denotes the Euler
characteristic. 

\textbf{~}\\ For a manifold with boundary, a non-vanishing vector
field \textit{transverse} to the boundary may or may not exist, depending
on the specific characteristics of the manifold and its boundary.
The condition for the existence of a non-vanishing vector field on
a differentiable manifold with boundary such that the vector field
is transverse to the boundary involves the notion of a vector field
being ``outward-pointing'' along the boundary~\cite{Tu}.

\textbf{~}\\ First, let us recall some general definitions~\cite{Jubin,Lee - Introduction to Smooth Manifolds,Tu}.

~
\begin{defn}
Let $M_{R}\cup\mathcal{H}$ be a manifold with boundary $\partial M_{R}=\mathcal{H}$
and $q\in\mathcal{H}$. A tangent vector $V_{q}\in T_{q}(M_{R}\cup\mathcal{H})$
is said to be inward-pointing if $V_{q}\notin T_{q}(\mathcal{H})$
and there is an $\epsilon>0$ and an associated curve $\gamma\colon[0,\epsilon)\longrightarrow M_{R}\cup\mathcal{H}$
such that $\gamma(0)=q$, $\gamma((0,\epsilon))\subset M^{\circ}$,
with $\gamma'(0)=V_{q}$. Correspondingly, we say a vector field $V_{q}\in T_{q}(M_{R}\cup\mathcal{H})$
is outward-pointing if $-V_{q}$ is inward-pointing.
\end{defn}

~
\begin{defn}
\label{def:collar-neighborhood}A collar of a manifold $M_{R}\cup\mathcal{H}$
with boundary $\partial M_{R}=\mathcal{H}$ is a diffeomorphism $\phi=(\phi_{1},\phi_{2})$
from an open neighborhood $U(\mathcal{H})$ of $\mathcal{H}$ to the
product $\mathbb{R}^{+}\times\mathcal{H}$ such that $\phi_{2}\mid_{\mathcal{H}}=\textrm{id}_{\mathcal{H}}$.
In particular, $\phi(\mathcal{H})=\{0\}\times\mathcal{H}$.
\end{defn}

\begin{rem}
\label{rem:Transverse + tangent component of V}Lee~\cite{Lee - Introduction to Smooth Manifolds}
refers to a different but equivalent definition: An open neighborhood
$U(\mathcal{H})$ of $\mathcal{H}$ is called a collar neighborhood
if it is the image of a smooth embedding $\phi^{-1}:[0,\varepsilon)\times\mathcal{H}\longrightarrow U(\mathcal{H})\subset(M_{R}\cup\mathcal{H})$
with $\varepsilon>0$.
\end{rem}

According to the Brown\textquoteright s collaring theorem~\cite{Brown - Collar Theorem,Lee - Introduction to Smooth Manifolds},
the boundary $\mathcal{H}$ has a collar neighborhood $U(\mathcal{H})$.
Given a vector field $V$ on $M_{R}\cup\mathcal{H}$ with a collar
$\phi$ we define on the open neighborhood $U(\mathcal{H})$ the tangent
and transverse components of $V$ with respect to $\phi$:\textbf{~}\\ 
\[
V_{\parallel}\coloneqq T\phi_{2}\circ V\colon U(\mathcal{H})\longrightarrow T(\mathcal{H})
\]

\[
V_{\pitchfork}\coloneqq T\phi_{1}\circ V\colon U(\mathcal{H})\longrightarrow\mathbb{R}.
\]

\textbf{~}\\ Moreover, a vector field is termed $0$-transverse if
it is transverse to the zero section of the tangent bundle.

~

\textbf{If $M_{R}\cup\mathcal{H}$ is a noncompact and connected manifold
with boundary:~}\\ In the case where $M_{R}\cup\mathcal{H}$ is connected
and noncompact with a boundary, a non-vanishing vector field transverse
to $\mathcal{H}$ can always be constructed using the following reasoning:

~

\uline{Local existence}: One standard way to see the local existence
of such a vector field is through the construction of a collar neighborhood
of the boundary. Using a collar neighborhood $U(\mathcal{H})$, we
can define a local vector field that is transverse to the boundary,
namely the vector field that points in the direction of the $[0,\varepsilon)$
coordinate in the collar: On $U(\mathcal{H})\subset(M_{R}\cup\mathcal{H})$
we use the collar coordinates to set $V(\phi^{-1}(t,\hat{y}))=\frac{\partial}{\partial t}\mid_{\phi^{-1}(t,\hat{y})}$,
where $\frac{\partial}{\partial t}$ is the vector field in the collar
direction, where $\hat{y}$ represents $\hat{y}=(y^{1},\ldots,y^{n-1})$. 

\textbf{~}\\ Note that in our setting, because $\textrm{Rad}_{q}\cap T_{q}\mathcal{H}=\{0\}$
$\forall q\in\mathcal{H}$, on $\mathcal{H}$, all non-spacelike vectors
are lightlike. These lightlike vectors are naturally divided into
two classes: those pointing towards $M_{L}$ and the those pointing
towards $M_{R}$. And instead of a collar neighborhood, we have an
$\mathcal{H}$-global \textcolor{black}{neighborhood $U=\bigcup_{q\in\mathcal{H}}U(q)$
of $\mathcal{H}$ }at our disposal~{\cite{Rieger-Embedding}.
Within this neighborhood $U$ the absolute time function $\mathfrak{h}(t,\mathbf{\hat{x}}):=t$,
with }$\mathbf{\hat{x}}=(x^{1},\ldots,x^{n-1})$, imposes a natural time direction by postulating that the future corresponds
to the increase of the absolute time function. Accordingly $\frac{\partial}{\partial t}$
serves as a vector field with an initial point on $\mathcal{H}$ that
points in the direction in which $t=\mathfrak{h}(t,\mathbf{\hat{x}})$ increases and $x_{i}$
is constant. 

~

\uline{Global existence}: For each point $p\in\mathcal{H}$ there
is a local chart $(U,\phi)$ mapping $p$ to a neighborhood in $\mathbb{H}^{n}$.
In each local chart $\phi:U\rightarrow\mathbb{H}^{n}$, we can construct
a vector field that points directly out of or into the boundary of
$\mathcal{H}$. Given a cover of $M_{R}\cup\mathcal{H}$ by such local
charts, we can use a smooth partition of unity $\{\rho_{i}\}$$\LyXZeroWidthSpace$
subordinate to this cover and piece together these local transverse
vector fields into a global vector field on $M_{R}\cup\mathcal{H}$.
This extension of these local constructions to a global vector field
on $M_{R}\cup\mathcal{H}$ can be done in a way that the resulting
vector field remains smooth and transverse to $\mathcal{H}$: The
global vector field $V$ on $M_{R}\cup\mathcal{H}$ can be constructed
as a weighted sum of the local vector fields $V_{i}$ using the partition
of unity functions $\rho_{i}$. Specifically, $V=\sum_{i}\rho_{i}V_{i}$.
Since each $V_{i}$ is transverse to the boundary in its chart and
the transition functions between charts preserve transversality, the
resulting global vector field $V$ is transverse to $\mathcal{H}$.

~

However, a transverse vector field on a manifold with boundary does
not recessarily have isolated zeros. For a vector field to have non-isolated
zeros, the zeros must form a set of positive measure (such as a continuum
of points). Whether a transverse vector field has non-isolated zeros
depends on the construction of the vector field and the topology of
the manifold.

~

\uline{Generic vector field}: In a typical construction of a transverse
vector field on a manifold with boundary, we can ensure that zeros,
if they exist, are isolated by slightly perturbing the vector field
by means of a generic Morse function~\cite{Guillemin + Pollack,Matsumoto}.
This is because generic vector fields on smooth manifolds tend to
have isolated zeros due to the transversality argument: if a vector
field $V$ is transverse to the zero section of the tangent bundle,
its zeros will be isolated. 

~

\uline{Remove isolated zeros}: If $M_{R}\cup\mathcal{H}$ is noncompact,
we can establish a compact exhaustion, denoted as $\varnothing=K_{0}\subset K_{1}\subset K_{2}\subset\ldots\subseteq M=\bigcup_{i}K_{i}$.
Zeros of a vector field existing in $K_{i}\setminus K_{i-1}$ are
systematically pushed to $K_{i+1}\setminus K_{i}$. Notably, this
process leaves the vector fields defined on $K_{i-1}$ unchanged.
By continuously pushing all zeros of a vector field towards infinity,
we obtain a well-defined nonvanishing vector field on $T(M_{R}\cup\mathcal{H})$,
where $T(M_{R}\cup\mathcal{H})$ represents a tangent bundle on $M_{R}\cup\mathcal{H}$.\textbf{~}\\{}

The argument presented above can be generalized to the case of a line
element field. Given a distinct transverse vector field without zeros
specified on the hypersurface $\mathcal{H}$, it follows that the
line element field associated with this vector field can also be extended
along the hypersurface to a zero-free line element field in the Riemannian
sector $M_{R}$. This is possible because a ``proper''
zero-free vector field exists in $M_{R}\cup\mathcal{H}$, from which
the line element field is derived by ``forgetting'' its orientation.
Essentially, this boils down to determining the existence of a non-trivial
section of the normal bundle of $\mathcal{H}$.

\textbf{~}\\{}

\textbf{If $M_{R}\cup\mathcal{H}$ is a compact and connected manifold
with boundary:~}\\ In the case where $M_{R}\cup\mathcal{H}$ is connected
and compact with a boundary, the situation is a bit more complicated.
The existence of a non-vanishing vector field transverse to the boundary
of a compact manifold $M_{R}\cup\mathcal{H}$ depends on the topology
and geometry of the manifold. One key result that addresses this question
is the \textit{generalized} Poincar\'{e}-Hopf theorem~\cite{Jubin}
(see~\cite{Milnor} for the original Poincar\'{e}-Hopf theorem).

\begin{rem}
This version of the Poincar\'{e}-Hopf theorem offers a generalization
that accommodates a boundary without requiring the vector field to
point outward, while considering the manifold's orientation. Consequently,
it extends the theorem to allow vector fields that do not necessarily
point outward at the boundary, ensuring that the vector field remains
transverse to the boundary. This is relevant to our request for a
non-vanishing vector field transverse to the boundary because it provides
the necessary theoretical framework to handle vector fields on manifolds
with boundaries, considering their indices and behaviors at the boundaries.
However, it does not guarantee the existence of such vector fields
in all cases (see, for instance, Example~\ref{exa: Vector field on B^2}
below).\textbf{ }Although the cited article by Jubin~\cite{Jubin}
has not been peer-reviewed, we consider it an appropriate reference.
It gives a self-contained presentation of the proof, which, together
with the explanation of special terminology, is scattered over various
articles, and provides related references.\textbf{~}\\{}
\end{rem}

Note that in the context of a vector field $V$ on a manifold $M$,
a zero of $V$ refers to a point $p$ in the manifold where the vector
field vanishes, meaning $V(p)=0$. So if the sum of the indices of
the zeros of a vector field $V$ on a compact manifold $M$ is $0$,
then the Poincar\'{e}-Hopf index theorem implies that the Euler characteristic
$\chi(M)=0$. However, having zero Euler characteristic does not provide
a direct guarantee of the existence of a non-vanishing vector field
transverse to the boundary. The Poincar\'{e}-Hopf theorem provides
a necessary condition for the existence of a nowhere vanishing vector
field. It states that if $M_{R}\cup\mathcal{H}$ is a compact manifold
with boundary and admits a nowhere vanishing vector field, then the
Euler characteristic $\chi(M_{R}\cup\mathcal{H})$ must be zero. So,
while $\chi(M_{R}\cup\mathcal{H})=0$ is a necessary condition, it
is not a sufficient condition for the existence of a nowhere vanishing
vector field on a compact manifold with boundary.

\textbf{~}\\ Moreover, compact manifolds with boundary and Euler
characteristic zero do not naturally exist (the term ``naturally''
suggests that such manifolds are not commonly encountered without
intentional construction or modification). The construction of such
manifolds involves combining closed manifolds in a specific manner---primarily by attaching handles to wellknown manifolds, such as a
torus with handles, real projective space with handles, and handlebodies
with boundaries.

\vspace{0.01\baselineskip}

\begin{example}
Consider a torus with handles, which can be thought of as a higher-genus
surface obtained by attaching handles to a torus. If the handles are
attached in a way that preserves the orientation of the torus, it's
possible to construct a non-vanishing vector field transverse to the
boundary. However, if the handles are attached in a way that reverses
the orientation of the torus, then constructing such a vector field
becomes impossible due to the Poincar\'{e}-Hopf theorem. Since the
Euler characteristic of the torus with handles is not zero, there
must be points where a non-vanishing vector field is tangent to the
boundary.\textbf{~}\\ Therefore, a non-vanishing vector field transverse
to the boundary is not guaranteed in the general case. And in summary,
the existence of a non-vanishing vector field transverse to the boundary
for a torus with handles depends on the specific way the handles are
attached and whether the resulting orientation is consistent with
the Poincar\'{e}-Hopf theorem.
\end{example}

\vspace{0.01\baselineskip}

\begin{example}
\label{exa: Vector field on B^2}Consider the unit disk $B^{2}$.
To be transverse at the boundary, a non-vanishing vector field should
consistently point outward. But we know~\cite{Munkres} that given
a non-vanishing vector field $V$ on $B^{2}$, there exists a point
of $S^{1}$ where the vector field $V$ points directly inward and
a point of $S^{1}$where it points directly outward. Thus we arrive
at a contradiction with the condition for transversality. The reason
is as follows:\textbf{~}\\ If the vector field $(x,V(x))$ on $B^{2}$
(written as an ordered pair) changes direction between pointing ``outward''
and ``inward'' along the boundary of a differentiable manifold,
it means that the field is not consistently transverse to the boundary,
but becomes tangent at some boundary points. In mathematical terms
this means that for some $x\in S^{1}$ we have $V(x)=ax$ for some
$a<0$, where $V(x)=ax$ for some $a>0$ means pointing directly outward~\cite{Tu}.
Hence, the vector field $V$ vanishes if $V(x)=0$. Hence, a vector
field on $B^{2}$ that is nowhere-vanishing cannot point inward everywhere
or outward everywhere, so it has to be tangent to the boundary somewhere
by continuity.\textbf{~}\\ In other words, if the vector field fails
to be transverse to the boundary this includes inconsistency in direction
across the boundary, violating the desired conditions for a well-defined
transverse vector field. Analogously this also applies to the associated
line element field $\{V,-V\}$.\textbf{~}\\ This example reflects
the situation for the $2$-dimensional case of the \textquoteleft no
boundary proposal\textquoteright{} spacetime.
\end{example}

In conclusion, when $M_{R}\cup\mathcal{H}$ is compact with a boundary,
the relationship between the Euler characteristic $\chi(M_{R}\cup\mathcal{H})$
and the existence of a nowhere vanishing vector field on a compact
manifold with boundary becomes more subtle. There are no clear conditions
that we can impose on $M_{R}\cup\mathcal{H}$ for the existence of
a nowhere vanishing vector field that is also transverse to the boundary.
As a consequence, we will either explicitly require the existence
of a non-vanishing vector field transverse to the boundary, or limit
our analysis exclusively to the non-compact case where constructing
such a vector field transverse to the boundary is always feasible.
Analogously, this also applies to the associated line element field
$\{V,-V\}$.

~

\subsubsection{Transformation Theorem}

In the context of the global Transformation Theorem, as opposed to
the local version, the emphasis is also placed on the region in the
Riemannian sector that is ``distant'' from the hypersurface. The
relationship between the Riemannian sector and the hypersurface is
crucial in determining whether the equivalence statement in the local
Transformation Theorem~\ref{Transformation-Theorem-(local)-1} is
applicable in the desired manner. The reasoning above yields now quite
easily the proof for the global Transformation Theorem~\ref{Transformation-Theorem-(global)-1}:

\begin{thm*}[Global Transformation Theorem, transverse radical]
Let $M$ be an transverse, signature-type changing manifold of $\dim(M)=n\geq2$,
which admits in $M_{R}\cup\mathcal{H}$ a smoothly defined non-vanishing
line element field that is transverse to the boundary $\mathcal{H}$.
Then the metric $\tilde{g}$ associated with a signature-type changing
manifold $(M,\tilde{g})$ is a transverse, type-changing metric with
a transverse radical if and only if $\tilde{g}$ is obtained from
a Lorentzian metric $g$ via the Transformation Prescription $\tilde{g}=g+fV^{\flat}\otimes V^{\flat}$,
where, for all $q\in \mathcal H \coloneqq f^{-1}(1)=\{p\in M\colon f(p)=1\}$,
\[
df(q)\neq0 \quad \text{and} \quad (df(V))(q)\neq0.
\]
\end{thm*}

\begin{proof}
Since $(M,g)$ is Lorentzian, there always exist a well-defined smooth
timelike line element field $\{V,-V\}$ on all of $M$. Consequently,
any triple $(g,V,f)$, where $g$ is a Lorentzian metric, $V\in\{V,\text{\textemdash}V\}$
a non-vanishing line element field on $M$ with $g(V,V)=-1$ and $f\colon M\longrightarrow\mathbb{R}$
a smooth function, yields a signature-type changing metric $\tilde{g}=g+fV^{\flat}\otimes V^{\flat}$,
which is defined over the entire manifold $M$ (see\textit{ }Propositon~\ref{Proposition Transformation-Prescription-1}).
Hence, the local version of the Transformation Theorem~\ref{Transformation-Theorem-(local)-1}
entails also the global version.
\end{proof}

\begin{rem}
Remember that the triples $(g,V,f)$ form equivalence classes (see
Proposition~\ref{prop:Equivalence-classes-(V,g,f)}), where all triples
within an equivalence class yield the same metric $\tilde{g}$. If
an element's properties change in a way that aligns with a different
equivalence class, it may move accordingly to a different equivalence
class, as elucidated in Subsection~\ref{subsec:Representation-of (g,V,f)}.
Furthermore, it is ruled out that, amid all these (described) perturbations
within an equivalence class, a transformation from an $f$ with $df\neq0$
to an $f$ with $df=0$ can occur. This is evident from the fact that
$g_{00}=-1$ (this is due to the specific choice of $V$), resulting
in $\tilde{g}_{00}=(1-f)g_{00}$ because of the Transformation Prescription.
Since $\tilde{g}$ remains unchanged within an equivalence class and,
thus (with initially arbitrary coordinates, but perturbed in such
a way as to preserve the shape of $V$), especially $\tilde{g}_{00}$
remains unchanged as a function of the respective coordinates, $f$
also does not change as a function of the respective coordinates (although
$\tilde{g}_{00}$ and $f$ do change as functions on $M$, but not
as functions of the respective coordinates).
\end{rem}

Instead of requiring the existence of a non-vanishing vector field
transverse to the boundary, we can limit our analysis exclusively
to the non-compact case where constructing such a vector field transverse
to the boundary is always feasible. This yields the following

\begin{cor}
Let $M$ be an $n$-dimensional transverse, signature-type changing
manifold with $M_{R}\cup\mathcal{H}$ non-compact. Then the metric
$\tilde{g}$ associated with a signature-type changing manifold $(M,\tilde{g})$
is a transverse, type-changing metric with a transverse radical if
and only if $\tilde{g}$ is obtained from a Lorentzian metric $g$
via the Transformation Prescription $\tilde{g}=g+fV^{\flat}\otimes V^{\flat}$,
where, for all $q\in \mathcal H \coloneqq f^{-1}(1)=\{p\in M\colon f(p)=1\}$,
\[
df(q)\neq0 \quad \text{and} \quad (df(V))(q)\neq0.
\]\end{cor}

\begin{example}
Consider the classic type of a spacetime $M$ with signature-type
change which is obtained by cutting an $S^{4}$ along its equator
and joining it to the corresponding half of a de Sitter space, see
Figure~\ref{fig:No-Boundary Modell}. This is the universe model
obeying the ``no boundary'' condition.
The $4$-dimensional half sphere is homeomorphic to a disk (of corresponding
dimension) and there exists a non-vanishing line element field. However,
any such non-vanishing line element field $V$ will be tangent to
the equator $\mathcal{H}$ (which is the surface of signature-type
change) at some point $q\in\mathcal{H}$, see Example~\ref{exa: Vector field on B^2}.
Hence, we cannot extend $V$ smoothly across the equator to the Lorentzian
sector because $V\in T_{q}\mathcal{H}$ and thus $\exists\,q\in\mathcal{H}$
such that $V(f_{q})=df_{q}(V)=0$. Therefore, the radical is tangent
at some $q\in\mathcal{H}$, and the \textit{global version} of the
Transformation Theorem cannot apply.
\end{example}

\begin{figure}[H]
\centering{}\includegraphics[scale=0.6]{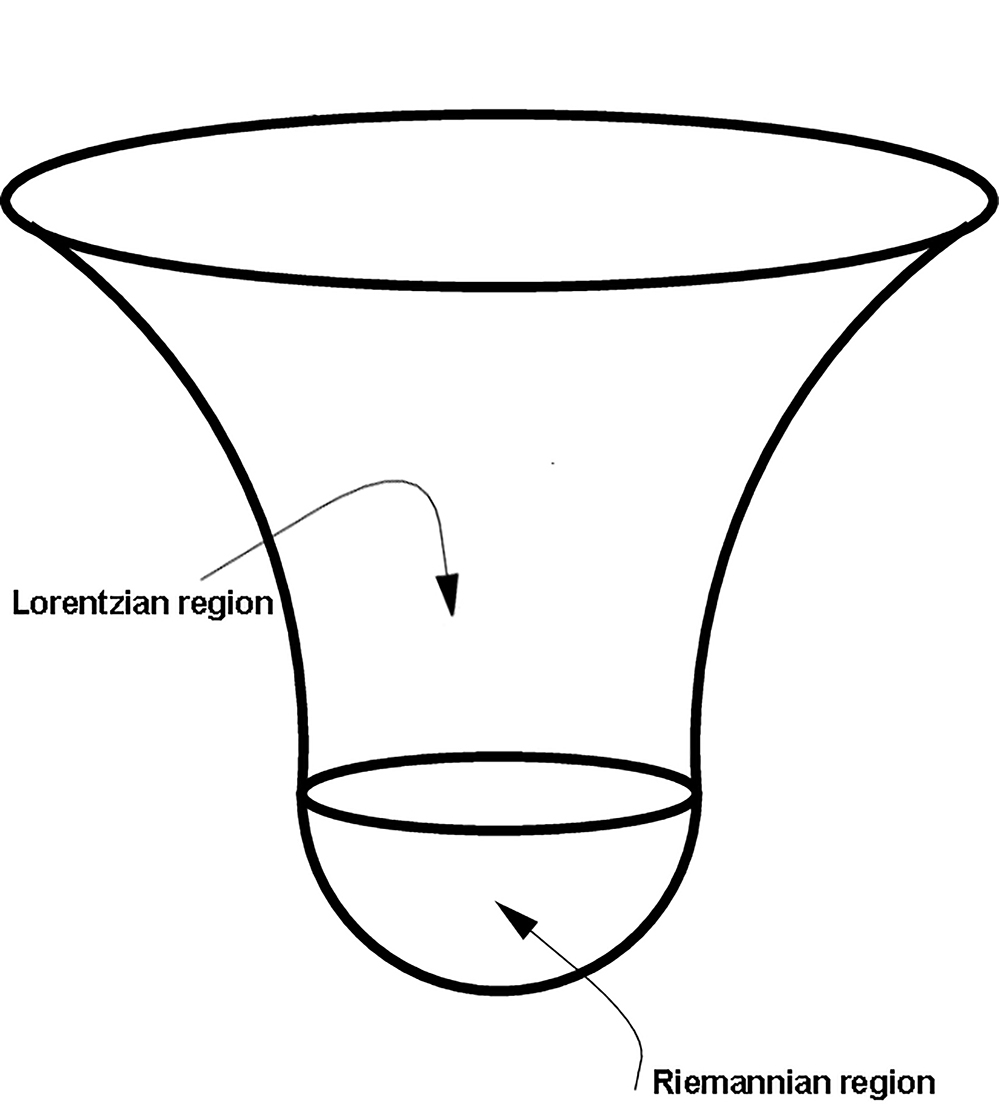}\caption{{\footnotesize{}\label{fig:No-Boundary Modell}}{\small{}Riemannian
and Lorentzian region in the Hartle-Hawking no-boundary model.}}
\end{figure}

\section{The Induced Metric on $\mathcal{H}$\label{sec: Induced metric}}

In the following we demonstrate that in general the induced metric
on the hypersurface $\mathcal{H}$ is either Riemannian or a positive
semi-definite pseudo metric.

\textbf{~}

Let $(M,\tilde{g})$ be an $n$-dimensional signature-type changing
manifold as introduced in Section~\ref{sec:Transformation-Prescription}
and $\mathcal{H}:=\{q\in M\!\!:g\!\!\mid_{q} is\;degenerate\}$ the
hypersurface of signature change. Furthermore, we assume again that
one component of $M\setminus\mathcal{H}$ is Riemannian and the other
one Lorentzian. Hence, $M\setminus\mathcal{H}$ is a union of two
semi-Riemannian manifolds with constant signature. As suggested by~\cite{Kossowksi + Kriele - Signature type change and absolute time in general relativity}
we consider the one-dimensional subspace of $T_{q}M$ defined as $\textrm{Rad}_{q}\coloneqq\{w\in T_{q}M\colon\tilde{g}(w,\centerdot)=0\}$,
for all $q\in\mathcal{H}$. Alternatively we can take the radical
as the kernel of the linear map $T_{q}M\longrightarrow T_{q}^{*}M\colon\,w\longmapsto(v\mapsto g(w,v))$.

\textbf{~}\\ If $v\in \textrm{Rad}_{q}$, then $\tilde{g}(v,\centerdot)=0$
must apply. Indeed, we have 
\[
\tilde{g}(v,\centerdot)=g(v,\centerdot)+\underset{1}{\underbrace{f(q)}}\underset{g(v,v)}{\underbrace{(v^{\flat}(v)}}\cdot\underset{g(v,\centerdot)}{\underbrace{v^{\flat}(\centerdot)}})=g(v,\centerdot)-g(v,\centerdot)=0.
\]
And because $\textrm{Rad}_{q}$ is a one-dimensional subspace of $T_{q}M$,
we have $\textrm{Rad}_{q}=\textrm{span}\{v\}$.

~
\begin{cor}
\label{par:Proposition. Closed and dense set}Let $(M,g)$ be an $n$-dimensional
signature-type changing manifold and $\mathcal{H}\subset M$ the hypersurface
of signature change. Then the set $\mathcal{H}$ is closed. Furthermore,
the set $M\setminus\mathcal{H}$ is dense in $M$ and open.
\end{cor}

\begin{proof}
That the set $\mathcal{H}$ is closed follows directly from the following
Definition~\ref{Definition. signature change}.
\end{proof}
~

Finally, we apply the Transformation Prescription (Proposition~\ref{Proposition Transformation-Prescription-1})
to prove Theorem~\ref{par: Theorem Rad_q-1},
\begin{thm*}
\label{par: Theorem Rad_q}If $q\in\mathcal{H}$ and $x\notin \textrm{Rad}_{q}$,
then $\tilde{g}(x,x)>0$ holds for all $x\in T_{q}M$.
\end{thm*}
\begin{proof}
We start by decomposing the vector $x$ into the sum of two components
with respect to a non-degenerate metric $g$, where one component
is parallel to $v$ and the other one is perpendicular to $v\colon$
$x=v^{\parallel}(x)+v^{\perp}(x)$ with $v^{\bigparallel}(x)=\frac{g(x,v)}{g(v,v)}v=-g(x,v)v$~\cite[p.50]{O'Neill}.
Recall that For $f(q)=1$ we have $\tilde{g}(\centerdot,\centerdot)=g(\centerdot,\centerdot)+g(v,\centerdot)g(v,\centerdot)$
on $T_{q}M$. Substituting for $v^{\bigparallel}(x)$ and rearranging
the vector decomposition, produces $v^{\perp}(x)=x+g(v,x)v$. Plugging
$v^{\perp}(x)$ into the metric gives

\[
g(v^{\perp}(x),v^{\perp}(x))=g(x+g(x,v)v,x+g(v,x)v)
\]

\[
=g(x,x)+2g(x,g(x,v)v)+g(x,v)^{2}g(v,v)
\]

\[
=g(x,x)+2g(x,v)^{2}+g(x,v)^{2}\underset{-1}{\underbrace{g(v,v)}}=g(x,x)+g(x,v)^{2}=\tilde{g}(x,x).
\]

\textbf{~}\\ Note that because of $g(v,v)=-1$, the vector $v$ is
timelike and $g$ is a Lorentzian metric. If a nonzero vector in $M$
is orthogonal to a timelike vector, then it must be spacelike~\cite{Naber},
hence, $v^{\perp}(x)$ is spacelike. Moreover, as $x\notin \textrm{Rad}_{q}=\textrm{span}\{v\}$
we know that $v^{\perp}(x)\neq0$, and therefore $\underset{\tilde{g}(x,x)}{\underbrace{g(v^{\perp}(x),v^{\perp}(x))}}>0$.
\end{proof}
\textbf{~}

As mentioned in Section~\ref{sec:Transformation-Prescription}, for
every $q\in\mathcal{H}$, the tangent space $T_{q}\mathcal{H}$ is
the kernel of the map $df_{q}$. Hence, $\ker(df_{q})=T_{q}\mathcal{H}$
is a vector subspace of $T_{q}M$. Provided that the radical is not
a vector subspace of $T_{q}\mathcal{H}$, that is $\textrm{Rad}_{q}\nsubseteq T_{q}\mathcal{H}$,
then the induced metric on $\mathcal{H}$ is Riemannian. This follows
directly from the proof of Theorem~\ref{par: Theorem Rad_q} because
for all $x\in T_{q}\mathcal{H}\setminus\{0\}$ the restriction of
the metric on $T_{q}M$ to the subspace $T_{q}\mathcal{H}$ is positive
definite. In the event of the radical being a vector subspace of $T_{q}\mathcal{H}$,
that is $\textrm{Rad}_{q}\subseteq T_{q}\mathcal{H}\subseteq T_{q}M$, then
according to the definition of $\textrm{Rad}_{q}$ the induced metric on $\mathcal{H}$
is degenerate. But then based on Proof~\ref{par: Theorem Rad_q}
we also have $\tilde{g}(x,x)\geq0$ for all $x\in T_{q}\mathcal{H}$.
Ultimately this leads to the conclusion that the induced metric on
$\mathcal{H}$ is a positive semi-definite pseudo metric with the
signature $(0,\underset{(n-2)\,\text{times}}{\underbrace{+,\ldots,+}})$.

\textbf{~}\\ Dependent on whether $\textrm{Rad}_{q}\subset T_{q}\mathcal{H}$,
or alternatively whether $\textrm{Rad}_{q}\subset\ker(df_{q})=T_{q}\mathcal{H}$,
the induced metric on the hypersurface of signature change $\mathcal{H}$
can be either Riemannian (and non-degenerate) or a positive semi-definite
pseudo-metric with signature $(0,\underset{(n-2)\,\text{times}}{\underbrace{+,\ldots,+}})$.
The latter one is degenerate if $\ker(df_{q})=T_{q}\mathcal{H}=\textrm{Rad}_{q}$.

~
\begin{rem}
A pseudo metric is considered to be a field of symmetric bilinear
forms which do not need to be everywhere non-degenerate. The main
point of pseudo-metric spaces is that we cannot use our concept of
distance to distinguish between different points, forcing us to think
of things in terms of equivalence classes where points declared to
have zero distance are considered equivalent.
\end{rem}

\begin{proof}
In our example, we have $\textrm{Rad}_{q}=\textrm{span}\{v\}$. And therefore,

$\textrm{Rad}_{q}=\textrm{span}\{v\}\subset T_{q}\mathcal{H}$ 

$\iff\textrm{span}\{v\}\subset\ker(df_{q})$

$\iff\textrm{span}\{v\}\subset\{w\in T_{q}M:df_{q}(w)=0\}$

$\iff df_{q}(v)=0$

$\iff v(f)=0$.
\end{proof}
\begin{acknowledgement*}
NER is greatly indebted to Richard Schoen for generously welcoming
her into his research group and to Alberto Cattaneo for affording
her creative independence throughout the duration of this research
endeavor. Moreover, NER acknowledges the partial support of the SNF
Grant No. 200021-227719. This research was (partly) supported by the
NCCR SwissMAP, funded by the Swiss National Science Foundation.
\end{acknowledgement*}

\end{document}